\documentclass[a4paper, 10pt]{amsart}

\usepackage{amsmath, amsfonts, amssymb}
\usepackage{amsthm}
\usepackage{esint}
\usepackage[
bookmarks=false]{hyperref}
\usepackage{cite}
\usepackage{array}
\usepackage{amsaddr}
\usepackage{enumerate}
\usepackage{pgfplots}

\title{\large{Coercivity of Weighted Kohn Laplacians:\\ the case of model monomial weights in $\C^2$}}

\author{Gian Maria Dall'Ara}
\address{Scuola Normale Superiore, Pisa, Italy}
\email{gianmaria.dallara@sns.it}

\date{\today}

\newcommand{\R}{\mathbb{R}}
\newcommand{\N}{\mathbb{N}}
\newcommand{\Q}{\mathbb{Q}}

\newcommand{\C}{\mathbb{C}}

\newcommand{\dbar}{\overline{\partial}}

\newcommand{\be}{\begin{equation*}}
\newcommand{\ee}{\end{equation*}}
\newcommand{\bel}{\begin{equation}}
\newcommand{\eel}{\end{equation}}
\newcommand{\bee}{\begin{eqnarray*}}
\newcommand{\eee}{\end{eqnarray*}}
\newcommand{\eps}{\varepsilon}

\newtheorem{thm}{Theorem}
\newtheorem{lem}[thm]{Lemma}
\newtheorem{prop}[thm]{Proposition}
\newtheorem{dfn}[thm]{Definition}
\newtheorem{cor}[thm]{Corollary}

\begin{document}

\maketitle

\begin{abstract} The \emph{weighted Kohn Laplacian} $\Box_\varphi$ is a natural second order elliptic operator associated to a weight $\varphi:\C^n\rightarrow\R$ and acting on $(0,1)$-forms, which plays a key role in several questions of complex analysis (see, e.g., \cite{haslinger-book}).

We consider here the case of \emph{model monomial weights} in $\C^2$, i.e., 
\be
\varphi(z,w):=\sum_{(\alpha,\beta)\in\Gamma}|z^\alpha w^\beta|^2,
\ee where $\Gamma\subseteq \N^2$ is finite. Our goal is to prove coercivity estimates of the form\bel\label{mu}
\Box_\varphi\geq \mu^2,
\eel where $\mu:\C^n\rightarrow\R$ acts by pointwise multiplication on $(0,1)$-forms, and the inequality is in the sense of self-adjoint operators. We recently proved in \cite{dallara} how to derive from \eqref{mu} new pointwise bounds for the weighted Bergman kernel associated to $\varphi$.

Here we introduce a technique to establish \eqref{mu} with \be
\mu(z,w)=c(1+|z|^a+|w|^b) \qquad(a,b\geq0),\ee where $a,b\geq0$ depend (and are easily computable from) $\Gamma$. As a corollary we also prove that, for a wide class of model monomial weights, the spectrum of $\Box_\varphi$ is discrete if and only if the weight is \emph{not decoupled}, i.e. $\Gamma$ contains at least a point $(\alpha,\beta)$ with $\alpha\neq0\neq\beta$.

Our methods comprise a new \emph{holomorphic uncertainty principle} and linear optimization arguments.

\end{abstract}

\section{Introduction}

\subsection{Motivations and goal of the paper}
Since the work of H\"ormander \cite{hormander}, many results in several complex variables have been established where a key role is played by a (typically, plurisubharmonic) weight $\varphi:\C^n\rightarrow\R$. In particular, an effective way of estimating the Bergman kernel of a weakly pseudoconvex domain in $\C^{n+1}$ is to consider first the same problem on the model domain associated to a plurisubharmonic function $\varphi:\C^n\rightarrow\R$:
\be
\Omega_\varphi:=\{(z,z_{n+1})\in \C^{n+1}\ :\ \Im(z_{n+1})>\varphi(z)\},
\ee
which, after a reduction to the boundary and a Fourier transform in the $\Re(z_{n+1})$-variable, leads to the consideration of a \emph{weighted Bergman kernel} in $\C^n$ (see \cite{haslinger-bergmanszego}). The adjective ``weighted" refers to the fact that the underlying measure is $e^{-2\varphi}$ times Lebesgue measure (see Section \ref{pointwise}). Typically one works under some finite-type assumption on domains, and hence on weights, the prototypical case being when $\varphi$ is a plurisubharmonic non-harmonic polynomial.

Many papers are dedicated to Bergman kernels on domains and weighted Bergman kernels on $\C^n$: the situation is well-understood when $n=1$ (see, e.g., \cite{nagel-rosay-stein-wainger} and \cite{christ}), and there are satisfactory results when $n\geq2$ if the domain or the weight satisfies some auxiliary assumption (e.g., it is convex \cite{mcneal-stein-bergman}, \cite{mcneal-stein-szego}, its complex Hessian has comparable eigenvalues \cite{koenig}, it is decoupled \cite{nagel-stein}, or geometrically separated \cite{charpentier-dupain}). 

Despite the rich literature, the case of a generic plurisubharmonic non-harmonic polynomial weight $\varphi:\C^n\rightarrow\R$ is far from being understood, and it is expected to raise delicate algebraic and geometric questions. A significant step in this direction has been taken recently by Nagel and Pramanik \cite{nagel-pramanik-diagonal}, who obtained pointwise bounds for the diagonal values of the Bergman kernel of $\Omega_\varphi$ when $\varphi$ is, in the author's terminology, a \emph{model monomial weight}, i.e., 
\be
\varphi(z_1,\dots,z_n):=\sum_{\alpha=(\alpha_1,\dots,\alpha_n)\in\Gamma}|z_1^{\alpha_1}\cdots z_n^{\alpha_n}|^2,
\ee
where $\Gamma\subseteq\N^n\setminus\{0\}$ is finite and non-empty. 

We recently proved in \cite{dallara} that pointwise bounds for the weighted Bergman kernel relative to $\varphi:\C^n\rightarrow \R$ can be obtained whenever the \emph{weighted Kohn Laplacian} $\Box_\varphi$, a second order elliptic operator naturally associated to $\varphi$ and acting on $(0,1)$-forms, is known to be $\mu$-coercive, i.e., \bel\label{mu2}
\Box_\varphi\geq \mu^2,
\eel where $\mu:\C^n\rightarrow\R$ acts by pointwise multiplication on $(0,1)$-forms, and the inequality is in the sense of self-adjoint operators (see Section \ref{kohn} for the precise definitions). The estimate we get for the weighted Bergman kernel is characterized by an exponential decay which depends quantitatively on $\mu$ (see Section \ref{pointwise} for details). We highlight the fact that coercivity conditions like \eqref{mu2} are often useful in the study of elliptic operators, and have other interesting consequences (see, e.g., Theorem \ref{model-disc} in the following). 

The goal of this paper is to establish \eqref{mu2} when $\varphi$ is a model monomial weight in $\C^2$, and with \be
\mu(z,w)=c(1+|z|^a+|w|^b) \qquad(a,b\geq0),\ee where $a,b\geq0$ depend (and are easily computable from) $\Gamma$. This is somehow the simplest class of weights that falls outside the scope of the existing literature. 

\subsection{Structure of the paper}
After defining rigorously weighted Kohn Laplacians, $\mu$-coercivity (Section \ref{kohn}), and model monomial weights (Section \ref{model}), we state our theorems in Section \ref{results}. Section \ref{pointwise} relies on \cite{dallara} to deduce estimates on the weighted Bergman kernel associated to model monomial weights in $\C^2$.

The proofs of our theorems are outlined in Section \ref{modelstruct}, and consist of two main ingredients: a linear optimization argument which exploits the specific algebraic nature of our weights (Section \ref{lambda-sec}), and a more general \emph{holomorphic uncertainty principle}, which we introduce in Section \ref{hol-sec}. The two ingredients are put together in Section \ref{est-sec}, where the proofs are concluded.

\subsection{Acknowledgements}

The present paper is part of the Ph.D. research the author conducted at Scuola Normale Superiore in Pisa, under the supervision of Fulvio Ricci (see \cite{dallara-thesis}). The author would like to thank him for the many useful discussions about the subject of this paper.

The author is very grateful to Alexander Nagel for teaching him a lot of harmonic and complex analysis during two visits to the University of Wisconsin, Madison, in 2011 and 2014, and in particular for many careful discussions about the results presented here.

\section{Definitions and statement of the results}

\subsection{Weighted Kohn Laplacians and $\mu$-coercivity}\label{kohn}

Let $\varphi:\C^n\rightarrow\R$ be a $C^2$ plurisubharmonic weight. This means that the \emph{complex Hessian} $H_\varphi=\left(\frac{\partial^2\varphi}{\partial z_j\partial \overline{z}_k}\right)_{j,k=1}^n$ is everywhere non-negative, i.e., \be
(H_\varphi(z)v,v)\geq0\qquad\forall z\in\C^n,\quad v\in\C^n.
\ee

We begin by introducing the weighted $L^2$ space
\be L^2(\C^n,\varphi):=\left\{f:\C^n\rightarrow\C\ :\ \int_{\C^n}|f|^2e^{-2\varphi}<+\infty\right\}.\ee
We denote by $L^2_{(0,q)}(\C^n,\varphi)$ the Hilbert space of $(0,q)$-forms with coefficients in $L^2(\C^n,\varphi)$.   
Since we will be working only with forms of degree less than or equal to $2$, we confine our discussion to these cases. Adopting the standard notation for differential forms, we have that $L^2_{(0,0)}(\C^n,\varphi)=L^2(\C^n,\varphi)$,\be
L^2_{(0,1)}(\C^n,\varphi):=\left\{u=\sum_{1\leq j\leq n} u_j d\overline{z}_j:\ u_j\in L^2(\C^n,\varphi)\quad\forall j\right\},
\ee and
\be
L^2_{(0,2)}(\C^n,\varphi):=\left\{w=\sum_{1\leq j<k\leq n} w_{jk}\ d\overline{z}_j\wedge d\overline{z}_k:\ w_{jk}\in L^2(\C^n,\varphi)\quad\forall j,k\right\}.
\ee
For the norms and the scalar products in these Hilbert spaces of forms we use the same symbols $||\cdot||_\varphi$ and $(\cdot,\cdot)_\varphi$, i.e., if $u,\widetilde{u}\in L^2_{(0,1)}(\C^n,\varphi)$, we have\be
||u||_\varphi^2=\sum_{1\leq j\leq n}||u_j||_\varphi^2,\quad (u,\widetilde{u})_\varphi=\sum_{1\leq j\leq n}(u_j,\widetilde{u}_j)_\varphi,
\ee while if $w,\widetilde{w}\in L^2_{(0,2)}(\C^n,\varphi)$, we have\be
||w||_\varphi^2=\sum_{1\leq j<k\leq n}||w_{jk}||_\varphi^2,\quad (w,\widetilde{w})_\varphi=\sum_{1\leq j<k\leq n}(w_{jk},\widetilde{w}_{jk})_\varphi.
\ee 
This ambiguity should not be a source of confusion.

We now introduce the initial fragment of the \emph{weighted $\dbar$-complex}:
\bel\label{dbar-complex}
L^2(\C^n,\varphi)\stackrel{\dbar}\longrightarrow L^2_{(0,1)}(\C^n,\varphi)\stackrel{\dbar}\longrightarrow L^2_{(0,2)}(\C^n,\varphi).
\eel
The symbol $\dbar$ denotes as usual both the operator $\dbar:L^2(\C^n,\varphi)\rightarrow L^2_{(0,1)}(\C^n,\varphi)$ defined on the domain \be
\mathcal{D}_0(\dbar):=\left\{f\in L^2(\C^n,\varphi):\ \frac{\partial f}{\partial \overline{z}_j}\in L^2(\C^n,\varphi)\ \forall j\right\}
\ee by the formula $\dbar f=\sum_j\frac{\partial f}{\partial \overline{z}_j}d\overline{z}_j$, and the operator $\dbar: L^2_{(0,1)}(\C^n,\varphi)\rightarrow L^2_{(0,2)}(\C^n,\varphi)$ defined on the domain \be
\mathcal{D}_1(\dbar):=\left\{u=\sum_ju_jd\overline{z}_j\in L^2_{(0,1)}(\C^n,\varphi):\ \frac{\partial u_k}{\partial \overline{z}_j}-\frac{\partial u_j}{\partial \overline{z}_k}\in L^2(\C^n,\varphi)\ \forall j,k\right\}
\ee by the formula $\dbar u=\sum_{j<k}\left(\frac{\partial u_k}{\partial \overline{z}_j}-\frac{\partial u_j}{\partial \overline{z}_k}\right)d\overline{z}_j\wedge d\overline{z}_k$. 

The weighted $\dbar$-complex \eqref{dbar-complex} is a complex, i.e., \bel\label{dbar-squared}
\dbar f\in \mathcal{D}_1(\dbar) \quad \text{and}\quad \dbar\dbar f=0\qquad\forall f\in \mathcal{D}_0(\dbar).
\eel 

Taking the Hilbert space adjoints of the operators in \eqref{dbar-complex} (as we can, since the operators are closed and densely defined), we have the dual complex:
\be
L^2(\C^n,\varphi)\stackrel{\dbar^*_\varphi}\longleftarrow L^2_{(0,1)}(\C^n,\varphi)\stackrel{\dbar^*_\varphi}\longleftarrow L^2_{(0,2)}(\C^n,\varphi).
\ee

We use the index $\varphi$ in the symbols for these operators to stress the fact that not only the domains $\mathcal{D}_1(\dbar^*_\varphi)\subseteq L^2_{(0,1)}(\C^n,\varphi)$ and $\mathcal{D}_2(\dbar^*_\varphi)\subseteq L^2_{(0,2)}(\C^n,\varphi)$, but also the formal expressions of $\dbar^*_\varphi$ depend on the weight $\varphi$. We omit these formulas, since they will play no direct role in what follows.

\begin{dfn}
The \emph{weighted Kohn Laplacian} is defined by the formula
\be
\Box_\varphi:=\dbar^*_\varphi\dbar+\dbar\dbar^*_\varphi
\ee 
on the domain 
\be
\mathcal{D}(\Box_\varphi):=\{u\in L^2_{(0,1)}(\C^n,\varphi):\ u\in\mathcal{D}_1(\dbar)\cap\mathcal{D}_1(\dbar^*_\varphi),\ \dbar u\in \mathcal{D}_2(\dbar^*_\varphi)\text{ and }\dbar^* _\varphi u\in \mathcal{D}_0(\dbar)\}.
\ee 
\end{dfn}

The weighted Kohn Laplacian is a densely-defined, closed, self-adjoint and non-negative operator on $L^2_{(0,1)}(\C^n,\varphi)$. The details of the routine arguments proving this fact can be found in \cite{haslinger-book}.

Finally, let us introduce the quadratic form \bel\label{energy}
\mathcal{E}_\varphi(u,v):=(\dbar u,\dbar v)_\varphi+(\dbar^*_\varphi u,\dbar^*_\varphi v)_\varphi,
\eel defined for $u,v\in\mathcal{D}(\mathcal{E}_\varphi):=\mathcal{D}_1(\dbar)\cap\mathcal{D}_1(\dbar^*_\varphi)$. Notice that, by definition of Hilbert space adjoints, \be
(\Box_\varphi u,v)_\varphi=\mathcal{E}_\varphi(u,v) \qquad \forall u\in \mathcal{D}(\Box_\varphi),\quad \forall v\in \mathcal{D}(\mathcal{E}_\varphi).
\ee
We will simply write $\mathcal{E}_\varphi(u)$ for $\mathcal{E}_\varphi(u,u)$. 

The well-known Morrey-Kohn-H\"ormander formula gives an alternative expression for $\mathcal{E}_\varphi(u)$. In order to state it, we identify the $(0,1)$-form $u=\sum_{j=1}^n u_jd\overline{z}_j$ with the vector field $u=(u_1,\dots,u_n):\C^n\rightarrow\C^n$, so that $(H_\varphi u,u)=\sum_{j,k=1}^n\frac{\partial^2\varphi}{\partial z_j\partial \overline{z}_k}u_j\overline{u}_k$. 
The Morrey-Kohn-H\"ormander formula is the following identity:
\bel\label{MKH}
\mathcal{E}_\varphi(u)=\sum_{j,k}\int_{\C^n}|\dbar_k u_j|^2e^{-2\varphi}+2\int_{\C^n} (H_\varphi u,u)e^{-2\varphi}\qquad\forall u\in\mathcal{D}(\mathcal{E}_\varphi).
\eel 
A proof may be found in \cite{haslinger-book} (or \cite{chen-shaw}, for the similar unweighted case). Identity \eqref{MKH} reveals the fundamental role played by $H_\varphi$ in the analysis of $\Box_\varphi$.

We conclude this section with the key notion of $\mu$-coercivity, that already appeared in our previous paper \cite{dallara}.

\begin{dfn}\label{coerc-def}
Given a measurable function $\mu:\C^n\rightarrow[0,+\infty)$, we say that $\Box_\varphi$ is \emph{$\mu$-coercive} if the following inequality holds
\bel\label{coerc-formula}
\mathcal{E}_\varphi(u)\geq ||\mu u||_\varphi^2 \qquad\forall u\in \mathcal{D}(\mathcal{E}_\varphi).
\eel
\end{dfn}

In view of \eqref{MKH}, $\mu$-coercivity is equivalent to the estimate\bel\label{MKH2}
\sum_{j,k}\int_{\C^n}|\dbar_k u_j|^2e^{-2\varphi}+2\int_{\C^n} (H_\varphi u,u)e^{-2\varphi}\geq \int_{\C^n}\mu^2|u|^2e^{-2\varphi}\qquad\forall u\in\mathcal{D}(\mathcal{E}_\varphi).
\eel

The concept of $\mu$-coercivity is a very natural one in the theory of elliptic operators. One can think of it as a spatially localized spectral gap condition (the usual spectral gap condition corresponds to the case when $\mu$ equals a positive constant). The following lemma shows how a qualitative information on the spectrum of $\Box_\varphi$ may be deduced from $\mu$-coercivity.

\begin{lem}\label{pp}
Assume that $\Box_\varphi$ is $\mu$-coercive for some $\mu$ such that \be
\lim_{z\rightarrow\infty}\mu(z)=+\infty.
\ee
Then the operator $\Box_\varphi$ has discrete spectrum.
\end{lem}

We recall that we say that a self-adjoint operator has discrete spectrum if its spectrum is a discrete subset of $\R$ consisting of eigenvalues of finite multiplicity.

\begin{proof}
This is essentially contained in \cite{haslinger-funct}, where the author proves that $\Box_\varphi$ admits a compact inverse $N_\varphi$ if and only if for every $\eps>0$ there exists $R<+\infty$ such that if \be
u\in\mathcal{D}(\mathcal{E}_\varphi)\quad\text{is such that}\quad\mathcal{E}_\varphi(u)\leq1,
\ee then
\be
\int_{|z|\geq R}|u|^2e^{-2\varphi}\leq \eps.
\ee
This condition is clearly equivalent to $\mu$-coercivity for some $\mu$ diverging at infinity. To conclude the proof, recall that a compact operator has discrete spectrum and that if the inverse of a self-adjoint operator has discrete spectrum, the same is true of the operator itself.
\end{proof}

\subsection{Model monomial weights}\label{model}

Let us specialize to $n=2$. We use $z$ and $w$ as coordinates on $\C^2$.

\begin{dfn}\label{model-dfn}
If $\Gamma\subseteq \N^2$ is finite, we define the \emph{model monomial weight} associated to $\Gamma$ as follows:\be
\varphi_\Gamma(z,w):=\sum_{(\alpha,\beta)\in\Gamma}|z^\alpha w^\beta|^2\qquad\forall (z,w)\in\C^2.
\ee

A model monomial weight is said to be \emph{decoupled} if $\Gamma\subseteq \N\times\{0\}\cup\{0\}\times\N$. 

A model monomial weight is said to be \emph{homogeneous} if there are $m,n\geq1$ for which the following two properties hold: \begin{enumerate}
\item $\{(m,0), (0,n)\} \subseteq\Gamma$,
\item every $(\alpha,\beta)\in\Gamma$ lies on the line segment connecting $(m,0)$ and $(0,n)$, i.e.\bel\label{model-segment}
n\alpha+m\beta=nm\qquad\forall (\alpha,\beta)\in\Gamma.
\eel
\end{enumerate}
\end{dfn}

Model monomial weights are sums of moduli squared of holomorphic functions, and thus they are plurisubharmonic.

Of course one could consider the analogous definitions in $\C^n$, associating a model monomial weight to any finite $\Gamma\subseteq \N^n$, but here we shall only treat the two-dimensional case.

Any homogeneous model monomial weight is homogeneous with respect to a system of not necessarily isotropic dilations, i.e.,
\be\varphi_\Gamma(t^\frac{1}{m} z, t^\frac{1}{n} w)=t^2\varphi_\Gamma(z,w)\qquad \forall t>0 \text{ and } (z,w)\in\C^2.\ee

In our analysis a key role will be played by the two quantities $\sigma$ and $\tau$ (depending on $\Gamma$) defined as the smallest non-negative real numbers such that \bel\label{model-sigma-tau}
\frac{1}{\sigma}\leq \frac{\beta}{\alpha}\leq \tau\qquad\forall (\alpha,\beta)\in\Gamma\text{ such that }\alpha,\beta\neq0.
\eel
One can choose two points $(\alpha_1,\beta_1)$ and $(\alpha_2,\beta_2)$ of $\Gamma$ such that 
\be \alpha_1=\sigma\beta_1\quad \text{and}\quad  \beta_2=\tau\alpha_2.\ee 
Notice that $\sigma,\tau<+\infty$. If $\varphi_\Gamma$ is decoupled, then $\sigma$ and $\tau$ are set to be equal to $0$, while if $\varphi_\Gamma$ is not decoupled both $\sigma$ and $\tau$ are always positive. 

\begin{center}

\begin{tikzpicture} 

\tiny

\begin{axis}[
axis x line=bottom, axis y line=left, 
xmin=0, xmax=17, ymin=0, ymax=17, 
xtick={0,4,8,12,16, 17}, ytick={0,3,6,9,12, 17},
xticklabels={0,4,8,12,16,$\alpha$}, yticklabels={0,3,6,9,12, $\beta$}, 
enlargelimits=false] 

\addplot[] coordinates {(0,12)}
[xshift=12pt, yshift=8pt]
        node {(0,12)}
;

\draw [fill] (0,120) circle [radius=1.5];

\addplot[] coordinates {(4,9)}
[yshift=8pt]
        node {$(\alpha_2,\beta_2)$=(4,9)}
;

\draw [fill] (40,90) circle [radius=1.5];

\addplot[] coordinates {(8,6)}
[xshift=12pt, yshift=8pt]
        node {(8,6)}
;

\draw [fill] (80,60) circle [radius=1.5];

\addplot[] coordinates {(12,3)}
[yshift=8pt]
        node {$(\alpha_1,\beta_1)$=(12,3)}
;

\draw [fill] (120,30) circle [radius=1.5];

\addplot[] coordinates {(16,0)}
[xshift=-12pt, yshift=8pt]
        node {(16,0)}
;

\draw [fill] (160,0) circle [radius=1.5];

\addplot[dotted] coordinates{(0,12)(4,9)(8,6)(12,3)(16,0)};
;
;

\end{axis}

\end{tikzpicture}
\end{center}

\begin{small}
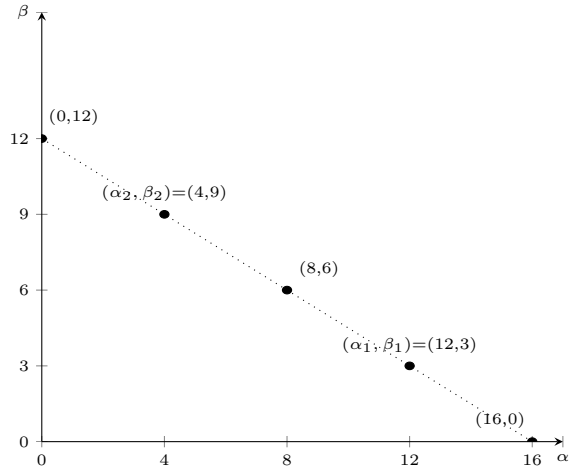
Figure 1: The plot of the set $\Gamma$ corresponding to the homogeneous model weight $\varphi_\Gamma(z,w)=|z|^{32}+|z|^{24}|w|^6+|z|^{16}|w|^{12}+|z|^8|w|^{18}+|w|^{24}$. In this case $\sigma=4$ and $\tau=\frac{9}{4}$.
\end{small}

\subsection{Our results}\label{results}

We can finally state our two main results. The first states that if a model monomial weight is homogenous, the associated Kohn Laplacian is $\mu$-coercive for a $\mu$ which can be easily computed from $\Gamma$.

\begin{thm}\label{model-thm}
Let $\varphi_\Gamma$ be a homogeneous model monomial weight. Then $\Box_{\varphi_\Gamma}$ is $\mu$-coercive, where \be
\mu(z,w)=c(1+|z|^\sigma+|w|^\tau).
\ee Here $c>0$ is a constant that depends on $\Gamma$.
\end{thm}

The second one is of a qualitative nature and concerns discreteness of the spectrum for an almost arbitrary model monomial weight.

\begin{thm}\label{model-disc}
If $\varphi_\Gamma$ is a model monomial weight for which there are $m,n\geq2$ such that $(m,0), (0,n)\in\Gamma$, then the Kohn Laplacian $\Box_\varphi$ has discrete spectrum if and only if $\varphi_\Gamma$ is not decoupled.
\end{thm}

This is, to the author's knowledge, an interesting new phenomenon: as soon as a mixed monomial is ``added" to a two-dimensional decoupled model monomial weight, the spectrum becomes discrete.

We plan to study the same questions in $\C^n$, when $n\geq3$. We expect that our arguments should require a significant effort to be generalized to more variables. Notice that it is not clear how should a generalization of Theorem \ref{model-disc} look like in $\C^3$: we know by work of Haslinger and Helffer (Theorem 6.1 of \cite{haslinger-helffer}) that $\Box_\varphi$ has not discrete spectrum when $\varphi(z)=|z_1|^{2m_1}+|z_2|^{2m_2}+|z_3|^{2m_3}$. Is it enough to ``add" a mixed term of the form $|z_1^{\alpha_1}z_2^{\alpha_2}z_3^{\alpha_3}|^2$ ($\alpha_1,\alpha_2,\alpha_3\neq0$) to the weight to make the spectrum discrete?

\subsection{Pointwise estimates of weighted Bergman kernels}\label{pointwise}

In \cite{dallara} (see also \cite{dallara-thesis}) we proved that an information about $\mu$-coercivity of $\Box_\varphi$ can be converted in a pointwise estimate of the weighted Bergman kernel with respect to $\varphi$. In order to state our result more precisely, we need to recall a few definitions. 

The \emph{weighted Bergman space} with respect to the $C^2$ plurisubharmonic weight $\varphi:\C^n\rightarrow\R$ is defined as follows:\be
A^2(\C^n,\varphi):=\left\{ h:\C^n\rightarrow\C:\ h\text{ is holomorphic and } h\in L^2(\C^n,\varphi)\right\}.
\ee

If $h\in A^2(\C^n,\varphi)$, then in particular it is harmonic and satisfies the mean value property $h(z)=\frac{1}{|B(z,r)|}\int_{B(z,r)}h$. The Cauchy-Schwarz inequality yields\bel\label{berg-bound}
|h(z)|\leq \frac{1}{|B(z,r)|}\sqrt{\int_{B(z,r)}e^{2\varphi}}\ ||h||_\varphi\qquad\forall h\in A^2(\C^n,\varphi),
\eel for any $z\in\C^n$ and $r>0$. This estimate has two elementary consequences:\begin{itemize}
\item[(a)] $A^2(\C^n,\varphi)$ is a closed subspace of $L^2(\C^n,\varphi)$ (by \eqref{berg-bound} convergence of a sequence of $A^2(\C^n,\varphi)$ in the $||\cdot||_\varphi$-norm implies uniform convergence, which preserves holomorphicity). We denote by $B_\varphi$ the orthogonal projector of $L^2(\C^n,\varphi)$ onto $A^2(\C^n,\varphi)$.
\item[(b)] The evaluation mappings $h\mapsto h(z)$ are continuous linear functionals of $A^2(\C^n,\varphi)$, and Riesz Lemma yields a function $K_\varphi:\C^n\times \C^n\rightarrow\C$ such that \be
h(z)=\int_{\C^n}K_\varphi(z,w)h(w)e^{-2\varphi(w)}d\mathcal{L}(w)\qquad\forall z\in\C^n,
\ee and $\overline{K_\varphi(z,\cdot)}\in A^2(\C^n,\varphi)$ for every $z\in\C^n$. 
\end{itemize}
The operator $B_\varphi$ is called the \emph{weighted Bergman projector} and the function $K_\varphi$ the \emph{weighted Bergman kernel} associated to the weight $\varphi$. It is immediate to see that \be
B_\varphi(f)(z)=\int_{\C^n}K_\varphi(z,w)f(w)e^{-2\varphi(w)}d\mathcal{L}(w)\qquad\forall z\in\C^n,
\ee for every $f\in L^2(\C^n,\varphi)$, i.e., $K_\varphi$ is the integral kernel of $B_\varphi$. Since $B_\varphi$ is self-adjoint, $K_\varphi(z,w)=\overline{K_\varphi(w,z)}$.

In \cite{dallara-thesis},  we introduced the class of \emph{admissible weights}. Those are the weights $\varphi$ such that:\begin{enumerate}
\item[(1)] the following $L^\infty$ doubling condition holds ($B(z,r):=\{z':\ |z'-z|<r\}$):\be
\sup_{B(z,2r)}\Delta\varphi\leq D\sup_{B(z,r)}\Delta\varphi \quad\forall z\in \C^n,\ r>0,
\ee for some finite constant $D$ which is independent of $z$ and $r$,
\item[(2)] there exists $c>0$ such that 
\bel\label{adm-lower}
\inf_{z\in\C^n}\sup_{B(z,c)}\Delta\varphi>0.
\eel
\end{enumerate}
If $\varphi$ is an admissible weight, the function\bel\label{radius}
\rho(z):=\sup\left\{r>0:\ \sup_{B(z,r)}\Delta\varphi\leq r^{-2}\right\}
\eel
is a \emph{radius function}, i.e., it is Borel and there exists a constant $C<+\infty$ such that for every $z\in\C^n$ we have\bel\label{radius-ineq}
C^{-1}\rho(z)\leq \rho(z')\leq C\rho(z) \qquad \forall z'\in B(z,\rho(z)).
\eel
See Section 4 of \cite{dallara} for details. It may be interesting to point out that the function $\rho$ defined by \eqref{radius} satisfies the following approximate identity when $\varphi$ is a polynomial:\bel\label{poly-approx}
\rho(z)\approx \min\left\{\left|\frac{\partial^{\alpha+\beta}\Delta\varphi(z)}{\partial z^\alpha\partial \overline{z}^\beta}\right|^{-\frac{1}{|\alpha|+|\beta|+2}}\ \colon\ \alpha,\beta\in\N^n\right\},
\eel where the implicit constant depends on the degree of $\varphi$ and $n$. To prove \eqref{poly-approx}, one can observe that $\sup_{B(z,r)}\Delta\varphi\leq r^{-2}$ is equivalent to $\sum_{\alpha,\beta}r^{2+|\alpha|+|\beta|}\frac{\partial^{\alpha+\beta}\Delta\varphi(z)}{\partial z^\alpha\partial \overline{z}^\beta}\lesssim 1$, as a simple Taylor expansion reveals.

We can finally state the main result of \cite{dallara}. 

\begin{thm}\label{pointwise-thm}
Let $\varphi$ be an admissible weight and assume that there exists \be\kappa:\C^n\rightarrow(0,+\infty)\ee such that:\begin{enumerate}
\item[\emph{(1)}] $\kappa$ is a bounded radius function, 
\item[\emph{(2)}] $\Box_\varphi$ is $\kappa^{-1}$-coercive.
\end{enumerate}
Then there exists $\eps>0$ such that the pointwise bound
\be
|K_\varphi(z,w)|\lesssim e^{\varphi(z)+\varphi(w)}\frac{\max\{\kappa(z),\rho(z)\}}{\rho(z)}\frac{e^{-\eps d(z,w)}}{\rho(z)^{n}\rho(w)^{n}}
\ee 
holds for every $z,w\in\C^n$, where $d(z,w)$ is the distance associated to the Riemannian metric 
\be
\frac{\sum_{j=1}^ndx_j^2+dy_j^2}{\max\{\kappa(z),\rho(z)\}^2}\qquad (z_j=x_j+iy_j).
\ee
\end{thm}

We refer to the paper for the proof and a deeper discussion of this result.

To see that the information contained in Theorem \ref{model-thm} can be plugged in Theorem \ref{pointwise-thm}, we prove now the following two claims:\begin{enumerate}
\item model monomial weights are admissible, 
\item $(1+|z|^a+|w|^b)^{-1}$ is a bounded radius function for every $a,b\geq0$.
\end{enumerate}

To verify the first claim, we begin by noticing that a model monomial weight $\varphi$ is a sum of squares of holomorphic functions, and hence it is $C^2$ and plurisubharmonic (alternatively, this follows from Proposition \ref{model-formulas} in Section \ref{lambda-sec}). 

Conditions (1) and (2) of the definition of admissibility only depend on the fact that $\Delta\varphi$ is a non-negative polynomial on $\C^n\equiv\R^{2n}$ ($n$ not necessarily equal to $2$). Let $d\in\N$ be the degree of this polynomial. The mappings\be
p\mapsto \sup_{|u|\leq 1}|p(u)|\quad\text{and}\quad p\mapsto \sup_{|u|\leq 2}|p(u)|
\ee 
are norms on the finite-dimensional vector space of real polynomials in $2n$ real variables of degree $\leq d$ on $\C^n\equiv\R^{2n}$, and therefore they are equivalent. In particular ($z\in\C^n$)\bee
\sup_{|u-z|\leq 2r}\Delta\varphi(u)&=&\sup_{|u|\leq 2}\Delta\varphi(z+ru)\\
&\leq& D\sup_{|u|\leq1}\Delta\varphi(z+ru)=D\sup_{|u-z|\leq r}\Delta\varphi(u).
\eee This proves condition (1).

As $z$ varies in $\C^n$, the polynomial $\Delta\varphi(z+\cdot)$ varies on a hyperplane not containing the origin of the vector space of real polynomials in $2n$ real variables of degree $\leq d$. To see this, just notice that any of the coefficients of a monomial of highest degree of $\Delta\varphi$ is not affected by translations. Since $p\mapsto \sup_{|u|\leq1}|p(u)|$ is a norm, we have\be
\inf_{z\in\C^n}\sup_{|u-z|\leq1}\Delta\varphi(u)=\inf_{z\in\C^n}\sup_{|u|\leq1}\Delta\varphi(z+u)>0,
\ee that is condition (2). This concludes the proof that $\varphi$ is an admissible weight.

To prove the second claim, just notice that the following stronger statement holds (if $\kappa(z,w):=(1+|z|^a+|w|^b)^{-1}$):  
\be
C^{-1}\kappa(z_0,w_0)\leq\kappa(z,w)\leq C\kappa(z_0,w_0)\qquad\forall (z,w)\in B((z_0,w_0),1).
\ee
In fact, it is equivalent to the elementary estimate \be
(1+|z|^a+|w|^b)\leq C(1+|z_0|^a+|w_0|^b)\qquad\forall (z,w)\in B((z_0,w_0),1).
\ee 

Thus, one can apply Theorem \ref{pointwise-thm} to deduce new pointwise estimates for $K_\varphi$ when $\varphi$ is a homogeneous model monomial weight. See also the proof of Theorem \ref{model-disc}, where we establish a weaker $\mu$-coercivity bound with $\mu$ of the form $1+|z|^\delta+|w|^\delta$. By the considerations above, this gives pointwise bounds for weighted Bergman kernels associated to more general model monomial weights.

\section{Outline of the proofs}\label{modelstruct}

Proving Theorem \ref{model-thm} and Theorem \ref{model-disc} boils down to establishing $\mu$-coercivity of certain weighted Kohn Laplacians for an appropriate $\mu$. This is clear for Theorem \ref{model-thm}, while it requires a little discussion in the case of Theorem \ref{model-disc}. 

First of all, the ``only if" part of Theorem \ref{model-disc} follows from work of Haslinger and Helffer: Theorem 6.1 of \cite{haslinger-helffer} states that if the weight is decoupled, then the spectrum of the weighted Kohn Laplacian is not discrete (this works for a wide class of weights including polynomial ones). Thus we are reduced to proving the ``if" part. In view of Lemma \ref{pp}, it is enough to show that if $\varphi$ is a non-decoupled model monomial weight with $(m,0)$, $(0,n)\in\Gamma$ for some $m,n\geq2$, then $\Box_\varphi$ is $\mu$-coercive for some $\mu$ diverging at infinity.

Thus, recalling \eqref{MKH2}, our goal is the estimate
\bel\label{MKH3}
\sum_{j=1}^2\left(\int_{\C^2}\left|\frac{\partial u_j}{\partial \overline{z}}\right|^2e^{-2\varphi}+\int_{\C^2}\left|\frac{\partial u_j}{\partial \overline{w}}\right|^2e^{-2\varphi}\right)+2\int_{\C^2} (H_\varphi u,u)e^{-2\varphi}\gtrsim \int_{\C^2}\mu^2|u|^2e^{-2\varphi}\eel
for every $u\in\mathcal{D}(\mathcal{E}_\varphi)$, in two cases:
\begin{enumerate}
\item[(a)] when $\varphi$ is a homogeneous model monomial weight and $\mu=1+|z|^\sigma+|w|^\tau$ (with $\sigma$ and $\tau$ as in Section \ref{model}),
\item[(b)] when $\varphi$ is a non-decoupled model monomial weight with $(m,0)$, $(0,n)\in\Gamma$ for some $m,n\geq2$, and $\mu$ diverging at infinity (dependent on $\varphi$).
\end{enumerate}

Notice that we use the notation $A\gtrsim B$ to denote $A\geq cB$ for a constant $c$, which is allowed to depend only on the weight.\newline

We now introduce the function:
\be
\lambda_\Gamma(z,w):=\min_{v\in\C^2\setminus\{0\}}\frac{(H_{\varphi_\Gamma}(z,w)v,v)}{|v|^2},
\ee
which equals the minimal eigenvalue of the complex Hessian of $\varphi_\Gamma$, and set as our goal the estimate for scalar-valued functions\bel\label{MKH4}
\int_{\C^2}\left|\frac{\partial u}{\partial \overline{z}}\right|^2e^{-2\varphi}+\int_{\C^2}\left|\frac{\partial u}{\partial \overline{w}}\right|^2e^{-2\varphi}+2\int_{\C^2} \lambda_\Gamma|u|^2e^{-2\varphi}\gtrsim \int_{\C^2}\mu^2|u|^2e^{-2\varphi}\qquad\forall u\in C^\infty_c(\C^2),
\eel
with $\varphi$ and $\mu$ as described above. The deduction of \eqref{MKH3} from \eqref{MKH4} is a simple approximation argument, which we omit.

The first consequence of \eqref{MKH4} is that $\Box_{\varphi_\Gamma}$ is $\mu$-coercive for $\mu\approx\sqrt{\lambda_\Gamma}$. Unfortunately, this is not enough to establish our theorems: for example, it will be clear later that $\lambda_\Gamma$ never diverges at infinity. This situation has to be compared to that in the theory of Schr\"odinger operators where the operator $-\Delta+V$ has discrete spectrum even if the potential $V:\R^d\rightarrow[0,+\infty)$ does not diverge at infinity. Discreteness of the spectrum of $-\Delta+V$ is in fact well-known (see, e.g. \cite{iwatsuka}) to be equivalent to the energy estimate\bel\label{schrod}
\int_{\R^d}|\nabla u|^2+\int_{\R^d}V|u|^2\gtrsim \int_{\R^d}\mu^2|u|^2,
\eel
for some $\mu$ diverging at infinity. This fact has to be compared with Lemma \ref{pp}. In this context, it is in virtue of the uncertainty principle that we expect \eqref{schrod} to hold for some $\mu^2$ larger than $V$. 

Inspired by this similarity, we look for a \emph{holomorphic uncertainty principle} that may serve an analogous purpose for our problem. Notice that the left hand of \eqref{MKH4} differs from that of \eqref{schrod} in two relevant aspects: the presence of the weight and the nature of the ``kinetic term", which contains only the barred derivatives and is therefore weaker. 

To turn these ideas into actual proofs, we proceed as follows:
\begin{enumerate}
\item In Section \ref{lambda-sec} we begin by showing that $\lambda_\Gamma$ is comparable to a rational function of $|z|^2$ and $|w|^2$ that can be computed from $\Gamma$ (Section \ref{det-tr}). This is the part where we use the specific nature of model monomial weights. Then, thanks to this approximate formula and a linear optimization argument (Section \ref{lin-prog}), we split $\C^2$ in regions where $\lambda_\Gamma$ is bounded from below by different functions of the form $|z|^{2a}|w|^{2b}$, where $a,b\in\Q$. If the weight is homogeneous we obtain sharp estimates (and as a consequence the statement of Theorem \ref{model-thm} is quantitative in nature). 
\item In Section \ref{hol-sec} we introduce our \emph{holomorphic uncertainty principle} to take care of the regions where $\lambda_\Gamma$ is too small.
\item Finally, in Section \ref{est-sec} we prove \eqref{MKH4}: outside of a hyperbolic neighborhood of the complex coordinate axes whose shape is dictated by the weight $\varphi$, we use the estimates of Section \ref{lambda-sec}, while on this neighborhood we exploit the holomorphic uncertainty principle.
\end{enumerate}

We would like to highlight the fact that the holomorphic uncertainty principle works for weights that are not necessarily polynomial, and in fact we believe that some more general formulation of it may hold and be useful for other problems as well.

\section{Estimating $\lambda_\Gamma$}\label{lambda-sec}

\subsection{Approximate formula for $\lambda_\Gamma$}\label{det-tr} Since $\varphi_\Gamma(z,w)=\sum_{(\alpha,\beta)\in\Gamma}|z|^{2\alpha} |w|^{2\beta}$, model weights only depend on the squared moduli of the coordinates. In view of this, we introduce the polynomial \bel\label{model-pol}
p_\Gamma(x,y):=\sum_{(\alpha,\beta)\in\Gamma}x^\alpha y^\beta\qquad (x,y)\in \R_{\geq0}^2,
\eel and in what follows we think of $x$ and $y$ both as independent non-negative variables and as denoting $|z|^2$ and $|w|^2$ respectively, so that 
\be\varphi_\Gamma(z,w)=p_\Gamma(|z|^2,|w|^2)=p_\Gamma(x,y).\ee This ambiguity will not be a source of confusion.

We now prove a very useful formula for the determinant and the trace of the complex Hessian $H_{\varphi_\Gamma}$ of a model monomial weight. In order to state it, we associate to any $\Gamma\subseteq \N^2$ four further subsets of $\N^2$:
\bee
\Gamma_r&:=&\{(\alpha,\beta)\in\Gamma\colon\ \alpha\neq0\} \quad\text{ ($r$ stands for ``right")},\\
\Gamma_u&:=&\{(\alpha,\beta)\in\Gamma\colon\ \beta\neq0\} \quad\text{($u$ stands for ``upper")},\\
\Gamma^{(1)}&:=&\{(\alpha,\beta)+(\gamma,\delta)\colon\ (\alpha,\beta), (\gamma,\delta)\in\Gamma \text{ linearly independent}\}-(1,1),\\
\Gamma^{(2)}&:=&\left[\Gamma_r-(1,0)\right]\cup\left[\Gamma_u-(0,1)\right].
\eee
Here $\Gamma_r-(1,0)$ denotes the collection $\{(\alpha-1,\beta):\ (\alpha,\beta)\in\Gamma_r\}$, and the other symbols have analogous meanings. Observe that if $(\alpha,\beta)$ and $(\gamma,\delta)$ are linearly independent elements of $\N^2$, then $(\alpha+\gamma-1,\beta+\delta-1)\in\N^2$, and hence $\Gamma^{(1)}\subseteq \N^2$.

\begin{prop}\label{model-formulas} If $\Gamma\subseteq \N^2$ is finite, then
\bee
\text{det}(H_{\varphi_\Gamma}(z,w))&\approx& \varphi_{\Gamma^{(1)}}(z,w),\\
\text{tr}(H_{\varphi_\Gamma}(z,w))&\approx& \varphi_{\Gamma^{(2)}}(z,w),
\eee where the implicit constants depend only on $\Gamma$. 
\end{prop}

In particular, this proposition shows that the model monomial weight $\varphi_\Gamma$ is weakly plurisubharmonic (i.e., $\lambda_\Gamma$ vanishes) on the set where $\varphi_{\Gamma^{(1)}}$ vanishes. Since $\varphi_{\Gamma^{(1)}}$ is itself a monomial model weight, this set may be easily determined from $\Gamma^{(1)}$, and may be empty, the origin, a complex coordinate axes ($\{z=0\}$ or $\{w=0\}$), or $\{z=0\}\cup\{w=0\}$. We omit the elementary details.

\begin{proof}
Let $h_1,\cdots,h_N:\C^2\rightarrow \C$ be holomorphic functions and consider the weight \be
\varphi:=\sum_{j=1}^N |h_j|^2.
\ee We have\be
\partial_z\dbar_z\varphi = \sum_j |\partial_zh_j|^2,\quad \partial_w\dbar_w\varphi = \sum_j |\partial_wh_j|^2,
\ee and \be
\partial_z\dbar_w\varphi = \sum_j \partial_zh_j\overline{\partial_w h_j},\quad \partial_w\dbar_z\varphi = \sum_j \partial_wh_j\overline{\partial_z h_j}.
\ee Hence\bee
\text{det}(H_\varphi) &=& \partial_z\dbar_z\varphi \cdot\partial_w\dbar_w\varphi - \partial_z\dbar_w\varphi\cdot\partial_w\dbar_z\varphi\\
&=& \sum_{j,k} |\partial_zh_j|^2|\partial_wh_k|^2 - \sum_{j,k} \partial_zh_j\overline{\partial_w h_j}\partial_wh_k\overline{\partial_z h_k}\\
&=& \frac{1}{2}\left(\sum_{j,k} |\partial_zh_j|^2|\partial_wh_k|^2 + |\partial_wh_j|^2|\partial_zh_k|^2 - 2\Re(\partial_zh_j\overline{\partial_w h_j}\partial_wh_k\overline{\partial_z h_k})\right)\\
&=& \frac{1}{2}\sum_{j,k} |\partial_zh_j\partial_wh_k - \partial_wh_j\partial_zh_k|^2.
\eee We also have\be
\text{tr}(H_\varphi)= \partial_z\dbar_z\varphi + \partial_w\dbar_w\varphi =\sum_j |\partial_zh_j|^2+|\partial_wh_j|^2.
\ee
Specializing to $\varphi_\Gamma(z,w):=\sum_{(\alpha,\beta)\in\Gamma}|z^\alpha w^\beta|^2$, we obtain (here l.i. stands for ``linearly independent"):
\bee
\text{det}(H_{\varphi_\Gamma}(z,w))&=&\frac{1}{2}\sum_{(\alpha,\beta),(\gamma,\delta)\in\Gamma}(\alpha\delta-\beta\gamma)^2|z^{\alpha+\gamma-1}w^{\beta+\delta-1}|^2\\
&\approx& \sum_{(\alpha,\beta),(\gamma,\delta)\in\Gamma \text{ l. i.}}|z^{\alpha+\gamma-1}w^{\beta+\delta-1}|^2\\
&\approx&\varphi_{\Gamma^{(1)}}(z,w),
\eee and \bee
\text{tr}(H_{\varphi_\Gamma}(z,w))&=&\sum_{(\alpha,\beta)\in\Gamma}\alpha^2|z^{\alpha-1}w^\beta|^2+\beta^2|z^\alpha w^{\beta-1}|^2\\
&\approx& \sum_{(\alpha,\beta)\in\Gamma\colon\alpha\neq0}|z^{\alpha-1} w^\beta|^2+\sum_{(\alpha,\beta)\in\Gamma\colon\beta\neq0}|z^\alpha w^{\beta-1}|^2\\
&\approx&\varphi_{\Gamma^{(2)}}(z,w).
\eee
It is easy to see that the implicit constants in the approximate equalities above depend only on $\Gamma$.
\end{proof}

Since $\det(H_{\varphi_\Gamma})$ equals the product of the eigenvalues of $H_{\varphi_\Gamma}$, and $\text{tr}(H_{\varphi_\Gamma})$ equals their sum, by Proposition \ref{model-formulas} we have\bel\label{lambda-formula}
\varphi_{\Gamma^{(1)}}(z,w)\approx \lambda_\Gamma(z,w)\cdot\varphi_{\Gamma^{(2)}}(z,w),
\eel
for any finite $\Gamma\subseteq \N^2$.

\subsection{A linear optimization argument to estimate $\lambda_\Gamma$}\label{lin-prog}
If $(u,v)\in\R^2$, we consider the curve in the non-negative quadrant $\R_{\geq0}^2$\be
C_{u,v}:\quad t\longmapsto (t^u,t^v)\qquad(t\geq1).
\ee
Notice that if $(u',v')$ is proportional to $(u,v)$, $C_{u,v}$ and $C_{u',v'}$ have the same range. If $A\subseteq \N^2$ is finite, using the notation \eqref{model-pol} we have\be
p_A(C_{u,v}(t))=\sum_{(\alpha,\beta)\in A} (t^u)^\alpha(t^v)^\beta=\sum_{(\alpha,\beta)\in A} t^{u\alpha+v\beta}\approx t^{m_{u,v}(A)}\quad(t\geq1),
\ee
where $m_{u,v}(A)$ is the maximum of the linear functional $(\xi,\eta)\mapsto u\xi+v\eta$ on the set $A\subseteq \R^2$, and the implicit constant depends on $\Gamma$ and is independent of $u,v$. 

We are interested in estimating $\lambda_\Gamma(z,w)$ when $(x,y)=(|z|^2,|w|^2)$ lies on the curve $C_{u,v}$. By formula \eqref{lambda-formula},\be
p_{\Gamma^{(1)}}(|z|^2,|w|^2)\approx\lambda_\Gamma(z,w) \cdot p_{\Gamma^{(2)}}(|z|^2,|w|^2).
\ee
If $(|z|^2,|w|^2)=(t^u,t^v)=C_{u,v}(t)$ ($t\geq1$), the discussion above gives\bel\label{lambda-zwt}
\lambda_\Gamma(z,w) \approx t^{m_{u,v}(\Gamma^{(1)})-m_{u,v}(\Gamma^{(2)})}.
\eel
We now present our analysis of this optimization problem first in the homogeneous case (Proposition \ref{lambda-prop}), where we obtain more precise results, and then in the more general case of the weights appearing in Theorem \ref{model-disc} (Proposition \ref{lambda-prop-2}).
\begin{prop}\label{lambda-prop} Let $\varphi_\Gamma$ be a homogeneous model monomial weight. Let $m,n, \sigma, \tau, \alpha_1,\beta_1,\alpha_2,\beta_2$ be as in Section \ref{model}. We define the three regions of $\C^2$:\bee
E_1&:=&\{|z|\geq 1,\quad |w|\leq |z|^\frac{m}{n}\},\\
E_2&:=&\{|w|\geq 1,\quad |w|^\nu\leq |z|\leq |w|^\frac{n}{m}\},
\eee and \be
E_3:=\{|w|\geq 1,\quad |z|\leq |w|^\nu\},
\ee where $\nu:=\frac{n-1-\beta_2}{\alpha_2-1}$. 
The following approximate identities hold:
\bee
\lambda_\Gamma(z,w)&\approx& |z|^{2\alpha_1}|w|^{2(\beta_1-1)}\qquad\forall (z,w)\in E_1,\\
\lambda_\Gamma(z,w)&\approx& |w|^{2(n-1)}\qquad\qquad\quad\forall (z,w)\in E_2,\\
\lambda_\Gamma(z,w)&\approx& |z|^{2(\alpha_2-1)}|w|^{2\beta_2}\qquad\forall (z,w)\in E_3.
\eee
\end{prop}

The figure below depicts the regions appearing in Proposition \ref{lambda-prop}.

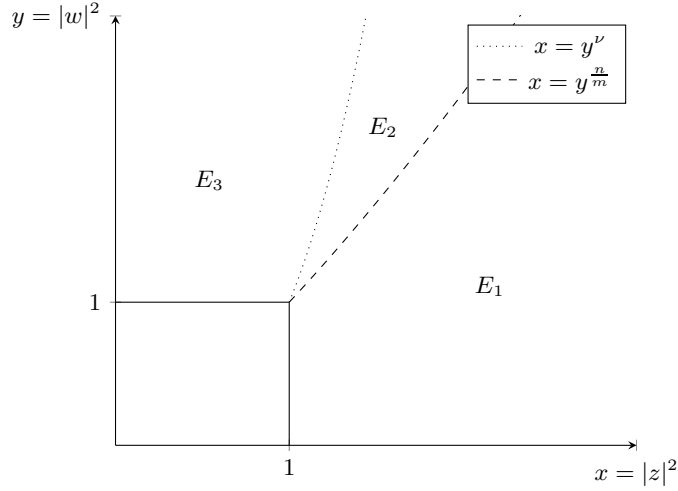
\begin{center}

\begin{tikzpicture} 
\small
\begin{axis}[
axis x line=bottom, axis y line=left, 
xmin=0, xmax=3, ymin=0, ymax=3, 
xtick={1,3},
ytick={1,3},
xticklabels={1,$x=|z|^2$},
yticklabels={1,$y=|w|^2$},
enlargelimits=false] 

\addplot[dotted, domain=1:3, samples=100]{x^3}
[xshift= 100pt,yshift=120pt]
node[pos=1]{$E_2$};
\addlegendentry{$x=y^\nu$}

\addplot[dashed, domain=1:3, samples=100]{x^1.3}
[xshift= 140pt,yshift=60pt]
node[pos=1]{$E_1$};
\addlegendentry{$x=y^\frac{n}{m}$}

\addplot[] coordinates {(1,0)(1,1)}
[xshift= 35pt,yshift=100pt]
node[pos=1]{$E_3$};

\addplot[] coordinates {(0,1)(1,1)}
[xshift= 32pt,yshift=26pt];

\end{axis}

\end{tikzpicture}

\end{center}

\begin{proof} Observe that for homogeneous model monomial weights the definitions of $\Gamma^{(1)}$ and $\Gamma^{(2)}$ take the slightly simpler forms\bee
\Gamma^{(1)}&:=&\{(\alpha+\gamma,\beta+\delta):\ (\alpha,\beta)\neq (\gamma,\delta)\in\Gamma\}-(1,1),\\
\Gamma^{(2)}&:=&\left(\Gamma\setminus \{(0,n)\}-(1,0)\right)\cup\left(\Gamma\setminus \{(m,0)\}-(0,1)\right).
\eee
Fix $(u,v)\in\R^2$ such that either $u$ or $v$ is positive: the union of the corresponding family of curves $C_{u,v}$ is $E_1\cup E_2\cup E_3$. By convexity considerations it is clear that the maximum of $u\xi+v\eta$ on $\Gamma$ is attained at $(m,0)$ if $um\geq vn$, while it is attained at $(0,n)$ if $um\leq vn$. We separately analyze the two cases, assuming without loss of generality that $m\geq n$.

\subsection*{Case I: $v\leq\frac{m}{n}u$} We have\be
m_{u,v}(\Gamma^{(1)})=um+m_{u,v}(\Gamma\setminus \{(m,0)\})-u-v,
\ee and 
\bee
m_{u,v}(\Gamma^{(2)})&=&\max\{m_{u,v}(\Gamma\setminus \{(0,n)\})-u,m_{u,v}(\Gamma\setminus \{(m,0)\})-v\}\\
&=&\max\{um-u,m_{u,v}(\Gamma\setminus \{(m,0)\})-v\}.
\eee
Hence \bel\label{lambda-max-formula}
m_{u,v}(\Gamma^{(1)})-m_{u,v}(\Gamma^{(2)})=\min\{m_{u,v}(\Gamma\setminus \{(m,0)\})-v, u(m-1)\}.
\eel
It is almost immediate to see that the maximum of $u\xi+v\eta$ on $\Gamma\setminus \{(m,0)\}$ is attained at the point $(\alpha_1,\beta_1)$ satisfying $\frac{\alpha_1}{\beta_1}=\sigma$, which we introduced above. Identity \eqref{lambda-max-formula} becomes\be
m_{u,v}(\Gamma^{(1)})-m_{u,v}(\Gamma^{(2)})=\min\{u\alpha_1+v\beta_1-v, u(m-1)\}.
\ee

The inequality $u\alpha_1+v\beta_1-v\leq u(m-1)$ holds if and only if\bel\label{lambda-condition-m-n}
v\leq \frac{m-1-\alpha_1}{\beta_1-1}u.
\eel This condition depends only on the ratio of $u$ and $v$, as it should. Observe that \be
\frac{m}{n}\leq\frac{m-1-\alpha_1}{\beta_1-1}.
\ee In fact, the inequality above is obviously equivalent to $mn-n\geq\alpha_1n+m\beta_1-m$, and recalling the homogeneity condition \eqref{model-segment} we see that this is the same as $m\geq n$, which we assumed before. This shows that condition \eqref{lambda-condition-m-n} is a consequence of $v\leq\frac{m}{n}u$ (since $u>0$ in this case) and thus \bel\label{lambda-caseI}
m_{u,v}(\Gamma^{(1)})-m_{u,v}(\Gamma^{(2)})=u\alpha_1+v\beta_1-v.
\eel

\subsection*{Case II: $u\leq\frac{n}{m}v$} Proceeding analogously to the case $um\geq vn$, this time formula \eqref{lambda-max-formula} is replaced by \be
m_{u,v}(\Gamma^{(1)})-m_{u,v}(\Gamma^{(2)})=\min\{m_{u,v}(\Gamma\setminus \{(0,n)\})-u, v(n-1)\},
\ee and the maximum of $ux+vy$ on $\Gamma\setminus \{(0,n)\}$ is attained at the point $(\alpha_2,\beta_2)$ satisfying $\frac{\beta_2}{\alpha_2}=\tau$. Hence\be
m_{u,v}(\Gamma^{(1)})-m_{u,v}(\Gamma^{(2)})=\min\{u\alpha_2+v\beta_2-u, v(n-1)\}.
\ee Here comes the difference with Case I: the minimum above equals $u\alpha_2+v\beta_2-u$ if and only if \be
u\leq\frac{n-1-\beta_2}{\alpha_2-1}v,
\ee but this condition is not automatically implied by the inequality $u\leq\frac{n}{m}v$. In fact \be
\frac{n-1-\beta_2}{\alpha_2-1}\leq \frac{n}{m},
\ee as may be easily verified using \eqref{model-segment} and the fact that $n\leq m$. Thus there are two further sub-cases: if \be
u\leq\frac{n-1-\beta_2}{\alpha_2-1}v
\ee then \bel\label{lambda-caseIIa}
m_{u,v}(\Gamma^{(1)})-m_{u,v}(\Gamma^{(2)})=u\alpha_2+v\beta_2-u, 
\eel while if \be
\frac{n-1-\beta_2}{\alpha_2-1}v\leq u\leq \frac{n}{m}v
\ee then \bel\label{lambda-caseIIb}
m_{u,v}(\Gamma^{(1)})-m_{u,v}(\Gamma^{(2)})=v(n-1). 
\eel

Putting \eqref{lambda-zwt}, \eqref{lambda-caseI}, \eqref{lambda-caseIIa}, and \eqref{lambda-caseIIb} together, we conclude the proof. \end{proof}

We state as a separate corollary the consequence of Proposition \ref{lambda-prop} that will be used in the proof of Theorem \ref{model-thm}.

\begin{cor}\label{cor-lambda}
Let $\varphi_\Gamma$ be a homogeneous model monomial weight. Let $\sigma$ and $\tau$ be as in Section \ref{model}. If we define the region of $\C^2$:\be
E:=\{|z|\geq1, |w|\geq |z|^{-\sigma}\}\cup\{|w|\geq1, |z|\geq |w|^{-\tau}\},
\ee we have\be
\lambda_\Gamma(z,w)\gtrsim |z|^{2\sigma}+|w|^{2\tau}\qquad\forall (z,w)\in E.
\ee
\end{cor}

The figure below depicts the region $E$ appearing in Proposition \ref{lambda-prop-2}. Notice that its complement contains two hyperbolic neighborhoods of the coordinate axes.

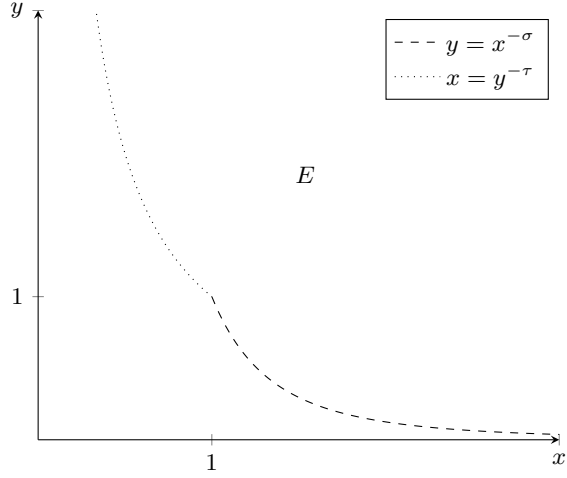
\begin{center}

\begin{tikzpicture} 
\small
\begin{axis}[
axis x line=bottom, axis y line=left, 
xmin=0, xmax=3, ymin=0, ymax=3, 
xtick={1,3},
ytick={1,3},
xticklabels={1,$x$},
yticklabels={1,$y$},
enlargelimits=false]

\addplot[dashed, domain=1:3, samples=100]{1/x^3}
[xshift= 80pt,yshift=10pt]
;
\addlegendentry{$y=x^{-\sigma}$}

\addplot[dotted, domain=0:1, samples=100]{1/x}
[xshift= 100pt,yshift=100pt]
node[pos=1]{$E$}
;
\addlegendentry{$x=y^{-\tau}$}



\end{axis}

\end{tikzpicture}

\end{center}

\begin{proof}
Let $E_1, E_2, E_3, \sigma, \tau, \alpha_1,\beta_1,\alpha_2,\beta_2,m,n,\nu$ be as in Proposition \ref{lambda-prop}.

Observe that $E_2\subseteq E\subseteq E_1\cup E_2\cup E_3$. 

If $(z,w)\in E\cap E_1$, then $|z|^\frac{m}{n}\geq|w|\geq |z|^{-\sigma}$ and, by Proposition \ref{lambda-prop},
\be
\lambda_\Gamma(z,w)\approx|z|^{2\alpha_1}|w|^{2(\beta_1-1)}\geq |z|^{2(\alpha_1-\sigma\beta_1+\sigma)}=|z|^{2\sigma}
\ee and \be
\lambda_\Gamma(z,w)\approx|z|^{2\alpha_1}|w|^{2(\beta_1-1)}\geq |w|^{2(\frac{n}{m}\alpha_1+\beta_1-1)}=|w|^{2(n-1)}.
\ee In the first identity we used the definition of $(\alpha_1,\beta_1)$, while in the second one we used \eqref{model-segment}. Notice that $\tau=\frac{\beta_2}{\alpha_2}\leq \beta_2\leq n-1$, and hence\bel\label{est-E1}
\lambda_\Gamma(z,w)\gtrsim |z|^{2\sigma}+|w|^{2\tau}\qquad\forall (z,w)\in E\cap E_1.
\eel

If $(z,w)\in E_2$, in particular $|z|\leq |w|^\frac{n}{m}$ and $|w|^{2(n-1)}\geq |z|^{2\frac{m}{n}(n-1)}$. Notice that $|w|\geq1$ on $E_2$ and $\nu\geq0$ and thus $|z|$ is also $\geq1$. If we show that $\frac{m}{n}(n-1)\geq \sigma$, we can then deduce that $|w|^{2(n-1)}\geq |z|^{2\sigma}$. To prove the inequality above one can plug in the identity $\sigma=\frac{\alpha_1}{\beta_1}$ and use \eqref{model-segment}. This, together with Proposition \ref{lambda-prop} and the already observed fact that $\tau\leq n-1$, allows to write that\bel\label{est-E2}
\lambda_\Gamma(z,w)\gtrsim|w|^{2(n-1)}\gtrsim |z|^{2\sigma}+|w|^{2\tau}\qquad\forall (z,w)\in E_2.
\eel

Finally, if $(z,w)\in E\cap E_3$ then in particular $|w|^\frac{n}{m}\geq|z|\geq |w|^{-\tau}$, and Proposition \ref{lambda-prop} yields\be
\lambda_\Gamma(z,w)\gtrsim|z|^{2(\alpha_2-1)}|w|^{2\beta_2}\geq |w|^{2(-\tau\alpha_2+\tau+\beta_2)}=|w|^{2\tau}
\ee and \be
\lambda_\Gamma(z,w)\gtrsim|z|^{2(\alpha_2-1)}|w|^{2\beta_2}\geq |z|^{2(\alpha_2-1+\frac{m}{n}\beta_2)}=|z|^{2(m-1)}.
\ee The last identity follows from \eqref{model-segment}. Since $\sigma=\frac{\alpha_1}{\beta_1}\leq \alpha_1\leq (m-1)$, we have\bel\label{est-E3}
\lambda_\Gamma(z,w)\gtrsim|z|^{2\sigma}+|w|^{2\tau}\qquad\forall (z,w)\in E\cap E_3.
\eel
Putting \eqref{est-E1}, \eqref{est-E2}, and \eqref{est-E3} together we obtain the thesis. 
\end{proof}

Let us proceed with the analysis of the weights appearing in Theorem \ref{model-disc}.

\begin{prop}\label{lambda-prop-2}
Let $\varphi_\Gamma$ be a non-decoupled model monomial weight such that $(m,0)$, $(0,n)\in\Gamma$ for some $m,n\geq2$. Let $\sigma$ and $\tau$ be as in Section \ref{model}. If we define the region of $\C^2$:\be
E:=\{|z|\geq1, |w|\geq |z|^{-\sigma}\}\cup\{|w|\geq1, |z|\geq |w|^{-\tau}\},
\ee we have\be
\lambda_\Gamma(z,w)\gtrsim |z|^{2\delta}+|w|^{2\delta}\qquad\forall (z,w)\in E,
\ee for some $\delta>0$ depending on $\Gamma$.
\end{prop}

As anticipated in Section \ref{modelstruct}, the bounds of Proposition \ref{lambda-prop-2} are not sharp in general, but they are sufficient for our purposes. 
\begin{proof}
We are going to show that if $(u,v)\in\R^2$ is such that $u\geq0$ and $v\geq-\sigma u$, then \bel\label{delta-bound}
m_{u,v}(\Gamma^{(1)})-m_{u,v}(\Gamma^{(2)})\geq \delta \max\{u,v\}.
\eel 
Recalling \eqref{lambda-zwt}, this proves that \be\lambda_\Gamma(z,w)\gtrsim |z|^{2\delta}+|w|^{2\delta}\ee in the region $\{|z|\geq1, |w|\geq |z|^{-\sigma}\}$. By symmetry, this also implies the same bound \eqref{delta-bound} in the region $\{|w|\geq1, |z|\geq |w|^{-\tau}\}$, and hence the statement of the Proposition.

To prove the claimed inequality, fix $(u,v)$ satisfying the assumptions above. We distinguish three cases depending on whether the maximum of $u\xi+v\eta$ on $\Gamma$ is attained on the $x$-axis, on the $y$-axis, or on $\left(\N\setminus\{0\}\right)^2$. In the analysis of the first two cases it will be useful to denote by $m$ and $n$ the largest natural numbers such that $(m,0)\in\N$ and $(0,n)\in\Gamma$. By assumption, $m,n\geq2$.

We denote by $\Gamma_{u,v}$ the subset of $\Gamma$ whose elements are not multiples of a fixed maximizer of $u\xi+v\eta$ on $\Gamma$. It is easy to see that \bel\label{max-formula}
m_{u,v}(\Gamma^{(1)})=m_{u,v}(\Gamma)+m_{u,v}(\Gamma_{u,v})-u-v.
\eel

\subsection*{Case I: $m_{u,v}(\Gamma)$ is attained on $\N\times\{0\}$} In this case $m_{u,v}(\Gamma)=mu$ and $\Gamma_{u,v}=\Gamma_u$ (recall the definition of $\Gamma_u$ in Section \ref{det-tr}). Moreover $m_{u,v}(\Gamma_r)=m_{u,v}(\Gamma)=mu$. By \eqref{max-formula}, we have\bee
&&m_{u,v}(\Gamma^{(1)})-m_{u,v}(\Gamma^{(2)})\\
&=&mu+m_{u,v}(\Gamma_u)-u-v-\max\{mu-u,m_{u,v}(\Gamma_u)-v\}\\
&=&\min\{m_{u,v}(\Gamma_u)-v,(m-1)u\}.
\eee
Inequality \eqref{delta-bound} is equivalent to the following four inequalities:\be
(m-1)u\geq\delta u, \quad (m-1)u\geq \delta v,\quad m_{u,v}(\Gamma_u)-v\geq\delta v,\quad m_{u,v}(\Gamma_u)-v\geq\delta u.
\ee

Since $m\geq2$, if we choose $\delta\leq 1$ the first inequality holds. Since the $m_{u,v}(\Gamma)$ is attained on the $x$-axis, we also have $mu\geq nv$, which implies the second one, if we choose $\delta\leq\frac{m-1}{m}n$.

To prove the third inequality, we distinguish two cases: $v\geq0$ and $v<0$. 

If $v\geq0$, we observe that $m_{u,v}(\Gamma_u)\geq nv$ (since $\{0\}\times\N\subseteq \Gamma_u$), and $nv\geq(1+\delta)v$, if we choose $\delta\leq n-1$ (which is a positive quantity, by the assumption $n\geq2$). 

If $v<0$, the inequality follows trivially from $m_{u,v}(\Gamma_u)\geq0$. To see this, recall the assumption $v\geq-\sigma u$. Since $\varphi_\Gamma$ is not decoupled, there is an element $(\alpha_1,\beta_1)\in\Gamma_u$ such that $\frac{\alpha_1}{\beta_1}=\sigma$. In particular, $m_{u,v}(\Gamma_u)\geq u\alpha_1+v\beta_1\geq0$.

We are left with the fourth inequality. We observe that \be
m_{u,v}(\Gamma_u)=\max_{(\alpha,\beta)\in\Gamma_u}u\alpha+\beta v
\ee is a continuous function of $v$ (for $u$ fixed) which is differentiable with derivative $\geq1$ outside a finite set. Since $\delta u+v$ has derivative $1$, it is enough to prove the fourth inequality when $v=-\sigma u$. By definition of $\sigma$, $m_{u,-\sigma u}(\Gamma_u)=0$, and our inequality becomes $\sigma u\geq \delta u$. Choosing $\delta\leq \sigma$, we are done.

This completes the analysis of the first case.

\subsection*{Case II: $m_{u,v}(\Gamma)$ is attained on $\{0\}\times\N$} In this case $v\geq0$, and necessarily $u\geq -\tau v$ (recall that $u\geq0$). Thus we can repeat the argument of Case I exchanging the role of the two variables $u$ and $v$.

\subsection*{Case III: $m_{u,v}(\Gamma)$ is attained on $\left(\N\setminus\{0\}\right)^2$} In this case $m_{u,v}(\Gamma)=m_{u,v}(\Gamma_r)=m_{u,v}(\Gamma_u)$, and 
\bee
&&m_{u,v}(\Gamma^{(1)})-m_{u,v}(\Gamma^{(2)})\\
&=&m_{u,v}(\Gamma)+m_{u,v}(\Gamma_{u,v})-u-v-\max\{m_{u,v}(\Gamma)-u, m_{u,v}(\Gamma)-v\}\\
&=&m_{u,v}(\Gamma_{u,v})-\max\{u,v\}.
\eee

Since $\Gamma_{u,v}$ contains the coordinate axis, $m_{u,v}(\Gamma_{u,v})\geq \max\{mu, nv\}$, and we can bound the expression above by $(\min\{m,n\}-1)\max\{u,v\}$. This completes the analysis of Case III, and hence the proof of \eqref{delta-bound}.

\end{proof}

\section{A holomorphic uncertainty principle}\label{hol-sec}

Corollary \ref{cor-lambda} and Proposition \ref{lambda-prop-2} allow to prove the $\mu$-coercivity estimate \eqref{MKH4} for every test function $u$ supported on the set 
\be
E:=\{|z|\geq1, |w|\geq |z|^{-\sigma}\}\cup\{|w|\geq1, |z|\geq |w|^{-\tau}\},
\ee
whenever either $\varphi$ is a homogeneous model monomial weight and $\mu=1+|z|^\sigma+|w|^\tau$, or $\varphi$ is a non-decoupled model monomial weight such that $(m,0),(0,n)\in\Gamma$ (for some $m,n\geq2$) and $\mu=1+|z|^\delta+|w|^\delta$ ($\delta>0$).

To take care of the complement of this region, and thus to complete the proof of Theorem \ref{model-thm} and Theorem \ref{model-disc}, we prove in this section the following lemma.

We denote by $D(z,r)$ the disc of center $z\in\C$ and radius $r$.

\begin{lem}\label{hol-unc} Let $V:D(z,r)\rightarrow[0,+\infty)$ be a measurable function and define \be
c:=\inf_{z'\in D(z,r)\setminus D\left(z,\frac{r}{2}\right)}V(z'). \ee
If $f\in L^2(D(z,r))$ is such that $\frac{\partial f}{\partial \overline{z}}\in L^2(D(z,r))$, then \bel\label{hol-ineq}
\int_{D(z,r)}\left|\frac{\partial f}{\partial \overline{z}}\right|^2+\int_{D(z,r)}V|f|^2\gtrsim \min\left\{c,\frac{1}{r^2}\right\}\int_{D(z,r)}|f|^2.
\eel
\end{lem}

The proof is based on a Poincar\'e-type inequality related to the $\frac{\partial}{\partial \overline{z}}$ operator and an elementary consequence of the Cauchy formula, which we now discuss. 

Put $D:=D(0,1)$. It is well-known (cf., e.g., \cite{chen-shaw}) that $\frac{\partial}{\partial \overline{z}}$ is solvable in $L^2(D)$, i.e., that if $g\in L^2(D)$ then there exists $f\in L^2(D)$ such that $\frac{\partial f}{\partial \overline{z}}=g$ and \be
\int_D|f|^2\lesssim \int_D|g|^2.
\ee
If $f\in L^2(D)$ is such that $\frac{\partial f}{\partial \overline{z}}\in L^2(D)$, the above solvability result yields $\widetilde{f}\in L^2(D)$ such that $f-\widetilde{f}$ is holomorphic and $\int_D|\widetilde{f}|^2\lesssim \int_D|\frac{\partial f}{\partial \overline{z}}|^2$. In particular, denoting by $B:L^2(D)\rightarrow L^2(D)$ the orthogonal projection onto the space of $L^2$ holomorphic functions (i.e., the unweighted Bergman projector of the unit disc), we have \bel\label{hol-poincare}
\int_D|f-B(f)|^2\leq \int_D|f-(f-\widetilde{f})|^2\lesssim \int_D\left|\frac{\partial f}{\partial \overline{z}}\right|^2.
\eel
This is the inequality we need. One should compare it with the usual Poincar\'e inequality in which $\frac{\partial}{\partial \overline{z}}$ is replaced by $\nabla$, and $B$ by $\frac{1}{|D|}\int_D$. Of course one could rescale the estimate to apply it to an arbitrary disc.

The second ingredient is the following inequality, which holds for every holomorphic function $h:D\rightarrow\C$: \bel\label{hol-cauchy}
\int_D|h|^2\lesssim \int_{D\setminus\frac{1}{2}D}|h|^2,
\eel and which follows easily from the Cauchy integral formula.

\begin{proof}[Proof of Lemma \ref{hol-unc}] By a trivial rescaling it is enough to prove the lemma for $z=0$ and $r=1$. Let then $V:D\rightarrow[0,+\infty)$ be such that $V\geq c$ on $D\setminus \frac{1}{2}D$, and $f\in L^2(D)$ be such that $\frac{\partial f}{\partial \overline{z}}\in L^2(D)$. If $\eps>0$ is a small parameter to be fixed later and we write $f=f_i+f_e$, where $f_e$ is zero on $\frac{1}{2}D$ and $f_i$ is zero on $D\setminus \frac{1}{2}D$, we have the following two possibilities:\begin{enumerate}
\item either $\int_D|f_e|^2\geq \eps\int_D|f_i|^2$, 
\item or $\int_D|f_e|^2< \eps\int_D|f_i|^2$.
\end{enumerate}

If (1) happens, a significant portion of the $L^2$ mass of $f$ is contained in the corona $D\setminus \frac{1}{2}D$ and \be
\int_D|f_e|^2\geq \frac{\eps}{2}\int_D|f_i|^2+\frac{1}{2}\int_D|f_e|^2\geq  \frac{\eps}{2}\int_D|f|^2.
\ee
Therefore\be
\int_D\left|\frac{\partial f}{\partial \overline{z}}\right|^2+\int_DV|f|^2\geq \int_{D\setminus \frac{1}{2}D}V|f|^2\geq c\int_D|f_e|^2\geq c\frac{\eps}{2}\int_D|f|^2,
\ee and \eqref{hol-ineq} holds.

If (2) happens, we use \eqref{hol-poincare}:\be
\int_D\left|\frac{\partial f}{\partial \overline{z}}\right|^2+\int_DV|f|^2\gtrsim \int_D|f-B(f)|^2.\ee
By the linearity of $B$ and condition (2), we have
\bee
\int_D|f-B(f)|^2&\geq&\frac{1}{2}\int_D|f_i-B(f_i)|^2-\int_D|f_e-B(f_e)|^2\\
&\geq&  \frac{1}{2}\int_D|f_i-B(f_i)|^2-\int_D|f_e|^2\\
&\geq&  \frac{1}{2}\int_D|f_i-B(f_i)|^2-\eps\int_D|f_i|^2.
\eee In the second line we used the fact that $1-B$ is an orthogonal projection. 

We claim that \bel\label{hol-orthogonal}
\int_D|f_i-B(f_i)|^2\geq a\int_D|f_i|^2,
\eel where $a$ is some small absolute constant.

Inequality \eqref{hol-orthogonal} immediately implies, choosing $\eps= \frac{a}{4}$, that \be
\int_D\left|\frac{\partial f}{\partial \overline{z}}\right|^2+\int_DV|f|^2\gtrsim\int_D|f_i|^2\gtrsim\int_D|f|^2.
\ee 
We are reduced to proving \eqref{hol-orthogonal}. In order to do this, we separate the two cases (for a new parameter $\delta$):\begin{enumerate}
\item $\int_{D\setminus \frac{1}{2}D}|B(f_i)|^2\geq\delta\int_D|f_i|^2$,
\item $\int_{D\setminus \frac{1}{2}D}|B(f_i)|^2<\delta\int_D|f_i|^2$.
\end{enumerate}

If (1) holds, \be
\int_D|f_i-B(f_i)|^2\geq\int_{D\setminus \frac{1}{2}D}|f_i-B(f_i)|^2=\int_{D\setminus \frac{1}{2}D}|B(f_i)|^2\geq\delta\int_D|f_i|^2.
\ee

If (2) holds instead, we apply \eqref{hol-cauchy} to the holomorphic function $B(f_i)$ to deduce that $\int_D|B(f_i)|^2\lesssim \delta\int_D|f_i|^2$. If we choose $\delta$ small enough we can write\be
\int_D|f_i-B(f_i)|^2\geq \frac{1}{2}\int_D|f_i|^2-\int_D|B(f_i)|^2\geq \frac{1}{4}\int_D|f_i|^2.
\ee
This concludes the proof of \eqref{hol-orthogonal}.
\end{proof}

Notice how the nature of uncertainty principle of the previous result is revealed by its proof: it shows that a function $f$ defined on a disc cannot be concentrated on a strictly smaller disc without having a large ``holomorphic kinetic energy" $\int_{D(z,r)}\left|\frac{\partial f}{\partial \overline{z}}\right|^2$.

One should also compare Lemma \ref{hol-unc} with the so-called Fefferman-Phong inequalities (see, e.g., \cite{fe-unc} or \cite{shen}), where $c$ is replaced by some kind of average of $V$ on the disc. Notice that one cannot hope for an improvement of Lemma \ref{hol-unc} of the form:\bel\label{false}
\int_{D(z,r)}\left|\frac{\partial f}{\partial \overline{z}}\right|^2+\int_{D(z,r)}V|f|^2\gtrsim \frac{1}{r^2}\min\left\{\int_{D(z,r)}V,1\right\}\int_{D(z,r)}|f|^2.
\eel
In fact, if $V\equiv 1$ on $D(0,\frac{1}{2})$ and $V\equiv0$ on $D(0,1)\setminus D(0,\frac{1}{2})$, we can test the hypothetical inequality\eqref{false} on $f(z)=z^m$ and obtain\be
\frac{4^{-m}}{m}\approx\int_{D(0,\frac{1}{2})}|z|^{2m}\gtrsim \int_{D(0,1)}|z|^{2m}\approx \frac{1}{m},
\ee which is a contradiction when $m$ tends to $+\infty$.

\section{Energy estimates}\label{est-sec}

Let $\varphi=\varphi_\Gamma$ be a model monomial weight. As in Section \ref{lambda-sec}, we put\be
E:=\{|z|\geq1, |w|\geq |z|^{-\sigma}\}\cup\{|w|\geq1, |z|\geq |w|^{-\tau}\}, 
\ee where $\sigma=\max_{(\alpha,\beta)\in\Gamma_u}\frac{\alpha}{\beta}$ and $\tau=\max_{(\alpha,\beta)\in\Gamma_r}\frac{\beta}{\alpha}$. 

\begin{prop}\label{key}
Assume that there are $a,b\geq0$ such that \bel\label{hyp}
\lambda_\Gamma(z,w)\gtrsim |z|^{2a}+|w|^{2b}\qquad \forall (z,w)\in E.
\eel
Then 
\be
\int_{\C^2}\left|\frac{\partial u}{\partial \overline{z}}\right|^2e^{-2\varphi}+\int_{\C^2}\left|\frac{\partial u}{\partial \overline{w}}\right|^2e^{-2\varphi}+2\int_{\C^2} \lambda_\Gamma|u|^2e^{-2\varphi}\gtrsim \int_{\C^2}(1+|z|^a+|w|^b)^2|u|^2e^{-2\varphi},
\ee for every $u\in C^\infty_c(\C^2)$.

\end{prop}

Putting Corollary \ref{cor-lambda}, Proposition \ref{lambda-prop-2}, and Proposition \ref{key} together, and recalling the discussion of Section \ref{modelstruct}, the reader can easily see that the proofs of Theorem \ref{model-thm} and Theorem \ref{model-disc} are completed.

\begin{proof}
We introduce the \emph{uncertainty regions} $U_0:=\{0\leq |z|, |w|\leq 2\}$,\be
U_r:=\{|z|> 1, 0\leq |w|\leq 2|z|^{-\sigma}\} \quad  \text{and}\quad U_u:=\{|w|> 1, 0\leq|z|\leq 2|w|^{-\tau}\}.
\ee

If $u\in C^\infty_c(\C^2)$ and $\Omega\subseteq\C^2$, we put
\be
F_\Omega(u):=\int_\Omega\left|\frac{\partial u}{\partial \overline{z}}\right|^2e^{-2\varphi}+\int_\Omega\left|\frac{\partial u}{\partial \overline{w}}\right|^2e^{-2\varphi}+2\int_\Omega \lambda_\Gamma|u|^2e^{-2\varphi}.\ee

We proceed by proving separately the estimate \bel\label{claim}
F_\Omega(u)\gtrsim\int_\Omega(1+|z|^a+|w|^b)^2|u|^2e^{-2\varphi},
\eel
when $\Omega=U_0$, $U_r$ and $U_u$, the case $\Omega=E$ being trivially implied by the hypothesis \eqref{hyp}. 

\subsection*{Estimate for $U_0$:}
By the trivial estimate $\sup_{(z,w)\in U_0}\varphi(z,w)\lesssim1$, we have
\be
F_{U_0}(u)\gtrsim\int_{U_0}\left(\left|\frac{\partial u}{\partial \overline{z}}\right|^2+\left|\frac{\partial u}{\partial \overline{w}}\right|^2\right)+\int_{U_0}\lambda_\Gamma |u|^2.
\ee 
Since $U_0=(2D)\times(2D)$, by Fubini's theorem and applying Lemma \ref{hol-unc} twice, we get ($d\mathcal{L}$ denotes Lebesgue measure):
\bee
&&\int_{(2D)\times(2D)}\left(\left|\frac{\partial u}{\partial \overline{z}}\right|^2+\left|\frac{\partial u}{\partial \overline{w}}\right|^2\right)+\int_{(2D)\times(2D)}\lambda_\Gamma |u|^2\\
&=& \int_{2D}\int_{2D}\left(\left|\frac{\partial u}{\partial \overline{z}}(z,w)\right|^2+\lambda_\Gamma(z,w) |u(z,w)|^2\right)d\mathcal{L}(z)d\mathcal{L}(w)\\
&+&\int_{(2D)\times(2D)}\left|\frac{\partial u}{\partial \overline{w}}\right|^2\\
&\gtrsim& \int_{2D}\left\{\left(\min_{1<|z'|\leq2}\lambda_\Gamma(z',w)\right)\int_{2D} |u(z,w)|^2d\mathcal{L}(z)\right\}d\mathcal{L}(w)\\
&+&\int_{(2D)\times(2D)}\left|\frac{\partial u}{\partial \overline{w}}\right|^2\\
&=&  \int_{2D}\int_{2D}\left(\left|\frac{\partial u}{\partial \overline{w}}(z,w)\right|^2+\left(\min_{1<|z'|\leq2}\lambda_\Gamma(z',w)\right) |u(z,w)|^2\right)d\mathcal{L}(w)d\mathcal{L}(z)\\
&\gtrsim&  \left(\min_{1<|z'|\leq2,1<|w'|\leq2}\lambda_\Gamma(z',w')\right) \int_{2D}\int_{2D}|u(z,w)|^2d\mathcal{L}(w)d\mathcal{L}(z).
\eee 
Notice that $\{1<|z'|\leq2,1<|w'|\leq2\}\subseteq E$ and hence \eqref{hyp} implies that the minimum above is $\approx1$. Thus \be
F_{U_0}(u)\gtrsim \int_{U_0}|u|^2\geq \int_{U_0}|u|^2e^{-2\varphi}\gtrsim\int_{U_0}(1+|z|^{2a}+|w|^{2b})|u|^2e^{-2\varphi},
\ee
where we used again the fact that $|z|,|w|\leq2$ on $U_0$.

\subsection*{Estimate for $U_r$:}
Notice that $\varphi=\varphi_{\Gamma_u}+\psi$, where $\psi$ is a function of $z$ alone.

By Fubini's theorem $F_{U_r}(u)\geq \int_{|z|\geq1}I(z)e^{-2\psi(z)}d\mathcal{L}(z)$, where  
\be
I(z):=\int_{D(0, 2|z|^{-\sigma})}\left(\left|\frac{\partial u}{\partial \overline{w}}(z,w)\right|^2+\lambda_\Gamma(z,w) |u(z,w)|^2\right)e^{-2\varphi_{\Gamma_u}(z,w)}d\mathcal{L}(w).
\ee

By the definition of $\sigma$, if $(z,w)\in U_r$ and $(\alpha,\beta)\in\Gamma_u$, then \be
|z^\alpha w^\beta|^2=\left(|z|^{\frac{\alpha}{\beta}}|w|\right)^{2\beta}\leq \left(|z|^\sigma|w|\right)^{2\beta}\lesssim 1,
\ee where in the first inequality we used the fact that $|z|\geq 1$. Summing over $(\alpha,\beta)\in\Gamma_u$, we obtain \be
\sup_{(z,w)\in U_r}\varphi_{\Gamma_u}(z,w)\lesssim 1.\ee
  
Using this bound and Lemma \ref{hol-unc} we obtain, for every $z$ of modulus greater than or equal to $1$,\bee
I(z)&\gtrsim&\int_{D(0, 2|z|^{-\sigma})}\left(\left|\frac{\partial u}{\partial \overline{w}}(z,w)\right|^2+\lambda_\Gamma(z,w) |u(z,w)|^2\right)d\mathcal{L}(w)\\
&\gtrsim& \left(\min_{|z|^{-\sigma}<|w'|\leq 2|z|^{-\sigma}}\lambda_\Gamma(z,w')\right)\int_{D(0, 2|z|^{-\sigma})}|u(z,w)|^2d\mathcal{L}(w).
\eee
Since the points $(z,w)$ such that $|z|\geq1$ and $|w|>|z|^{-\sigma}$ are contained in $E$, the hypothesis \eqref{hyp} gives\be
\min_{|z|^{-\sigma}<|w'|\leq 2|z|^{-\sigma}}\lambda_\Gamma(z,w')\approx \min_{|z|^{-\sigma}<|w'|\leq 2|z|^{-\sigma}} |z|^{2a}+|w'|^{2b}\approx |z|^{2a},
\ee for every $|z|\geq1$. We have\bee
F_{U_r}(u)&\gtrsim&\int_{|z|\geq1}\left(|z|^{2a}\int_{D(0, 2|z|^{-\sigma})}|u(z,w)|^2d\mathcal{L}(w)\right)e^{-2\psi(z)}d\mathcal{L}(z)\\
&\geq&\int_{|z|\geq1}\left(|z|^{2a}\int_{D(0, 2|z|^{-\sigma})}|u(z,w)|^2e^{-2\varphi_{\Gamma_u}(z,w)}d\mathcal{L}(w)\right)e^{-2\psi(z)}d\mathcal{L}(z)\\
&\gtrsim&\int_{U_r}|z|^{2a}|u(z,w)|^2e^{-2\varphi(z,w)}d\mathcal{L}(z,w)\\
&\gtrsim&\int_{U_r}(1+|z|^{2a}+|w|^{2b})|u(z,w)|^2e^{-2\varphi(z,w)}d\mathcal{L}(z,w).
\eee
The last step follows from the inequalities $|w|\leq 1$ and $|z|\geq1$, which hold for any $(z,w)\in U_r$.

\subsection*{Estimate for $U_u$:} This is done in complete analogy with the estimate for $U_r$, exchanging the role played by $z$ and $w$, and replacing $\sigma$ with $\tau$. 

The proof is complete.\end{proof}

\bibliographystyle{amsalpha}
\bibliography{model-weights}

\providecommand{\bysame}{\leavevmode\hbox to3em{\hrulefill}\thinspace}
\providecommand{\MR}{\relax\ifhmode\unskip\space\fi MR }
\providecommand{\MRhref}[2]{%
  \href{http://www.ams.org/mathscinet-getitem?mr=#1}{#2}
}
\providecommand{\href}[2]{#2}
\begin{thebibliography}{NRSW89}

\bibitem[CD14]{charpentier-dupain}
Ph. Charpentier and Y.~Dupain, \emph{Extremal bases, geometrically separated
  domains and applications}, Algebra i Analiz \textbf{26} (2014), no.~1,
  196--269. \MR{3234809}

\bibitem[Chr91]{christ}
Michael Christ, \emph{On the {$\overline\partial$} equation in weighted {$L^2$}
  norms in {${\bf C}^1$}}, J. Geom. Anal. \textbf{1} (1991), no.~3, 193--230.
  \MR{1120680 (92j:32066)}

\bibitem[CS01]{chen-shaw}
So-Chin Chen and Mei-Chi Shaw, \emph{Partial differential equations in several
  complex variables}, AMS/IP Studies in Advanced Mathematics, vol.~19, American
  Mathematical Society, Providence, RI; International Press, Boston, MA, 2001.
  \MR{1800297 (2001m:32071)}

\bibitem[Dal14]{dallara-thesis}
Gian~Maria Dall'Ara, \emph{Matrix {S}chr\"odinger operators and weighted
  {B}ergman kernels (ph.d. thesis)}, arXiv:1501.06311 (2014).

\bibitem[Dal15]{dallara}
\bysame, \emph{Pointwise estimates of weighted {B}ergman kernels in several
  complex variables}, arXiv:1502.00865 (2015).

\bibitem[Fef83]{fe-unc}
Charles~L. Fefferman, \emph{The uncertainty principle}, Bull. Amer. Math. Soc.
  (N.S.) \textbf{9} (1983), no.~2, 129--206. \MR{707957 (85f:35001)}

\bibitem[Has98]{haslinger-bergmanszego}
Friedrich Haslinger, \emph{The {B}ergman kernel functions of certain unbounded
  domains}, Ann. Polon. Math. \textbf{70} (1998), 109--115, Complex analysis
  and applications (Warsaw, 1997). \MR{1668719 (2000a:32006)}

\bibitem[Has11]{haslinger-funct}
\bysame, \emph{Compactness for the {$\overline\partial$}-{N}eumann problem: a
  functional analysis approach}, Collect. Math. \textbf{62} (2011), no.~2,
  121--129. \MR{2792515 (2012b:32048)}

\bibitem[Has14]{haslinger-book}
\bysame, \emph{The {$\overline{\partial}$}-{N}eumann problem and
  {S}chr{\"o}dinger operators}, de Gruyter Expositions in Mathematics, vol.~59,
  De Gruyter, Berlin, 2014. \MR{3222570}

\bibitem[HH07]{haslinger-helffer}
Friedrich Haslinger and Bernard Helffer, \emph{Compactness of the solution
  operator to {$\overline\partial$} in weighted {$L^2$}-spaces}, J. Funct.
  Anal. \textbf{243} (2007), no.~2, 679--697. \MR{2289700 (2007m:32022)}

\bibitem[H{\"o}r65]{hormander}
Lars H{\"o}rmander, \emph{$l^2$ estimates and existence theorems for the $\bar
  \partial $ operator}, Acta Math. \textbf{113} (1965), 89--152. \MR{0179443
  (31 \#3691)}

\bibitem[Iwa86]{iwatsuka}
Akira Iwatsuka, \emph{Magnetic {S}chr{\"o}dinger operators with compact
  resolvent}, J. Math. Kyoto Univ. \textbf{26} (1986), no.~3, 357--374.
  \MR{857223 (87j:35287)}

\bibitem[Koe02]{koenig}
Kenneth~D. Koenig, \emph{On maximal {S}obolev and {H}{\"o}lder estimates for
  the tangential {C}auchy-{R}iemann operator and boundary {L}aplacian}, Amer.
  J. Math. \textbf{124} (2002), no.~1, 129--197. \MR{1879002 (2002m:32061)}

\bibitem[MS94]{mcneal-stein-bergman}
J.~D. McNeal and E.~M. Stein, \emph{Mapping properties of the {B}ergman
  projection on convex domains of finite type}, Duke Math. J. \textbf{73}
  (1994), no.~1, 177--199. \MR{1257282 (94k:32037)}

\bibitem[MS97]{mcneal-stein-szego}
\bysame, \emph{The {S}zeg{\H o} projection on convex domains}, Math. Z.
  \textbf{224} (1997), no.~4, 519--553. \MR{1452048 (98f:32023)}

\bibitem[NP]{nagel-pramanik-diagonal}
Alexander Nagel and Malabika Pramanik, \emph{Diagonal estimates for the bergman
  kernel on certain domains in $\mathbb{C}^n$}, preprint.

\bibitem[NRSW89]{nagel-rosay-stein-wainger}
A.~Nagel, J.-P. Rosay, E.~M. Stein, and S.~Wainger, \emph{Estimates for the
  {B}ergman and {S}zeg{\H o} kernels in {${\bf C}^2$}}, Ann. of Math. (2)
  \textbf{129} (1989), no.~1, 113--149. \MR{979602 (90g:32028)}

\bibitem[NS06]{nagel-stein}
Alexander Nagel and Elias~M. Stein, \emph{The {$\overline{\partial}_b$}-complex
  on decoupled boundaries in {$\Bbb C^n$}}, Ann. of Math. (2) \textbf{164}
  (2006), no.~2, 649--713. \MR{2247970 (2007d:32036)}

\bibitem[She99]{shen}
Zhongwei Shen, \emph{On fundamental solutions of generalized {S}chr{\"o}dinger
  operators}, J. Funct. Anal. \textbf{167} (1999), no.~2, 521--564. \MR{1716207
  (2000j:35055)}

\end{thebibliography}

\end{document}